\documentclass{article}
\usepackage{amssymb,amsmath}
\usepackage{graphicx}

\def\qed{\hfill {\hbox{${\vcenter{\vbox{               %HOLLOW SQUARE
   \hrule height 0.4pt\hbox{\vrule width 0.4pt height 6pt
   \kern5pt\vrule width 0.4pt}\hrule height 0.4pt}}}$}}}

\def\bar{\overline}

\newtheorem{theorem}{Theorem}
\newtheorem{definition}{Definition}
\newtheorem{lemma}[theorem]{Lemma}
\newtheorem{proposition}[theorem]{Proposition}
\newtheorem{corollary}[theorem]{Corollary}
\newtheorem{example}{Example}

\newtheorem{remark}{Remark}

\newenvironment{proof}[1][Proof]{\smallskip\noindent{\bf #1.}\quad}%
{\qed\par\medskip}

\date{}

\title{\Large \textbf{Link invariants from finite biracks}
}

\author{Sam Nelson}

\begin{document}
\maketitle

\begin{abstract}
A \textit{birack} is an algebraic structure with axioms encoding
the blackboard-framed Reidemeister moves, incorporating quandles, 
racks, strong biquandles and semiquandles as special cases. In this 
paper we extend the counting invariant for finite racks to the case 
of finite biracks. We introduce a family of biracks generalizing 
Alexander quandles, $(t,s)$-racks, Alexander biquandles and 
Silver-Williams switches, known as $(\tau,\sigma,\rho)$-biracks. 
We consider enhancements of the counting invariant using writhe 
vectors, image subbiracks, and birack polynomials.
\end{abstract}

\medskip

\quad
\parbox{5in}{
\smallskip

}

\section{\large\textbf{Introduction}}

The modern study of algebraic structures with axioms corresponding to 
Reidemeister moves on knot diagrams goes back at least to the early 1980s 
with the more or less simultaneous work of Joyce \cite{J} and 
Matveev \cite{M}, drawing in some cases on previous work by
Takasaki \cite{T} and Conway and Wraith \cite{CW}. 

Assigning algebraic generators to the arcs in a knot diagram and interpreting 
crossings as operations paves the way for translating Reidemeister moves into 
algebraic axioms. 
Applying this formula to unoriented link diagrams yields the \textit{involutory 
quandle} or \textit{kei} structure; generalizing to oriented link diagrams
gives us the \textit{quandle} or \textit{distributive groupoid} structure.
Generalizing again to blackboard-framed diagrams as in \cite{FR} gives
us the \textit{rack} or \textit{automorphic set} structure.

A further generalization replaces \textit{arcs}, i.e., portions of the knot 
diagram from one undercrossing point to another, with \textit{semiarcs}, i.e.,
portions of the knot diagram from one over- or under-crossing point to the 
next. Semiarcs are edges in the graph obtained from a link diagram by 
replacing crossing points with vertices.
Previous work has been done on the semiarc-generated algebraic structures
arising from unframed oriented link diagrams, known as \textit{biquandles}
\cite{BF,FJK,KR,NV}. A special case of biquandles, applicable to \textit{flat 
virtual links} or \textit{virtual strings} and known as \textit{semiquandles},
is examined in \cite{HN} with further examples appearing in \cite{FT}.

In this paper we extend the method introduced in \cite{N} of obtaining
invariants of unframed classical and virtual knots and links from finite
racks to the case of 
\textit{biracks}, the algebraic structure generated by semiarcs
in a link diagram with axioms corresponding to blackboard-framed isotopy,
first introduced in \cite{FRS}.
The new family of invariants contains the quandle, rack and strong biquandle 
counting
invariants as special cases. In particular, the fundamental 
birack determines the knot quandle and fundamental rack and hence is 
complete invariant up to ambient homeomorphism for oriented framed knots 
and unsplit links; see \cite{FR, J, M}. 

%For practical computability, we define link invariants 
%using homomorphisms from the fundamental birack into finite
%biracks. 

The paper is organized as follows. In section \ref{S2} we define
biracks and give a few examples. In section \ref{S3} we  
describe a family of biracks containing Alexander quandles, 
Alexander biquandles, $(t,s)$-racks and Silver-Williams switches
as special cases. In section \ref{S4} we define a counting invariant of
unframed classical knots and links using finite biracks and
define enhancements of the birack 
counting invariants using image subbiracks, framing vectors and birack 
polynomials. We end with some questions for future research.

\section{\large\textbf{Birack basics}}\label{S2}

Let $X$ be a set.

\begin{definition}
\textup{We will say a map $B:X\times X\to X\times X$ is \textit{strongly 
invertible} provided $B$ satisfies the following three conditions:}
\begin{itemize}
\item \textup{$B$ is invertible, i.e there exists a map 
$B^{-1}:X\times X\to X\times X$
satisfying $B\circ B^{-1}=\mathrm{Id}_{X\times X}=B^{-1}\circ B$,}
\item \textup{$B$ is \textit{sideways invertible}, i.e there exists a unique 
invertible map $S:X\times X\to X\times X$ satisfying}
\[S(B_1(x,y),x)=(B_2(x,y),y),\]
\textup{for all $x,y\in X$, and}
\item \textup{The sideways maps $S$ and $S^{-1}$ are} \textit{diagonally 
bijective}, \textup{i.e. the compositions $  S_1^{\pm 1}\circ\Delta$, 
$S_2^{\pm 1}\circ\Delta$ of the components of $S$ and $S^{-1}$ with 
the map $\Delta:X\to X\times X$ defined by $\Delta(x)=(x,x)$ are bijections.} 
\end{itemize}
\end{definition}

\begin{definition}
\textup{A \textit{birack} $(X,B)$ is a set $X$ with a strongly
invertible map $B:X\times X\to X\times X$ which satisfies the
\textit{set-theoretic Yang-Baxter equation}}
\[(B\times \mathrm{Id})\circ(\mathrm{Id}\times B)\circ(B\times \mathrm{Id})=
(\mathrm{Id}\times B)\circ(B\times \mathrm{Id})\circ(\mathrm{Id}\times B)\]
\end{definition}

In previous work such as \cite{CN,KR,NV}, the components of $B$ were
interpreted as binary operations or right-actions of the set $X$ on 
itself, with $B_1(x,y)$ denoted by $y^x$, $B_2(x,y)=x_y$, 
$B_1^{-1}(x,y)=y_{\overline{x}}$ and $B_2^{-1}(x,y)=x^{\overline{y}}$.
We can write the Yang-Baxter requirement as three equations in the component 
maps:
\begin{eqnarray}
B_1(x,B_1(y,z))               & = & B_1(B_1(x,y),B_1(B_2(x,y),z)) \\
B_1(B_2(x,B_1(y,z)),B_2(y,z)) & = & B_2(B_1(x,y),B_1(B_2(x,y),z))\\
B_2(B_2(x,B_1(y,z)),B_2(y,z)) & = & B_2(B_2(x,y),z)
\end{eqnarray}
or in the notation of \cite{KR},
\[(z^y)^x=(z^{x_y})^{y^x},\quad (y_z)^{x_{z^y}}=(y^x)_{z^{x_y}},\quad 
\mathrm{and} \quad (x_{z^y})_{y_z}=(x_y)_z.\]

Recall that a \textit{blackboard-framed link} is an equivalence class
of link diagrams under the equivalence relation generated by the three 
\textit{blackboard-framed Reidemeister moves}, traditionally numbered 
according to the number of strands involved in the move:
\[
\begin{array}{c}
\includegraphics{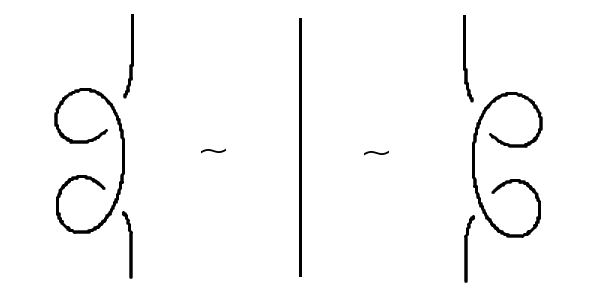} \\
\mathrm{type \ I} \end{array}
\quad \begin{array}{c}
\includegraphics{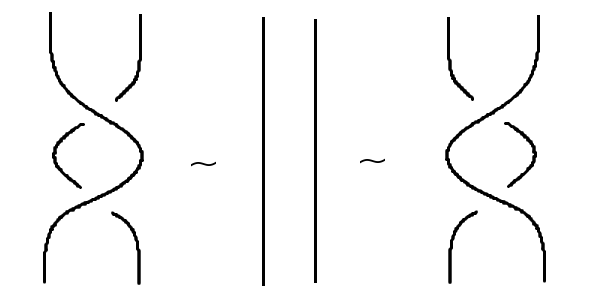}\quad \\
\mathrm{type \ II} \end{array}
\begin{array}{c}
\includegraphics{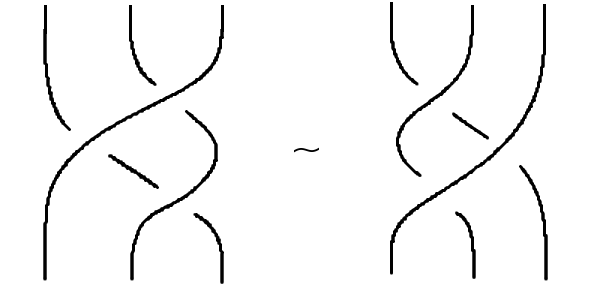}\quad \\
\mathrm{type \ III} \end{array}
\]

The birack axioms are chosen such that any labeling of the semiarcs
in an oriented blackboard-framed link diagram with elements of $X$ satisfying 
the identifications
\[\includegraphics{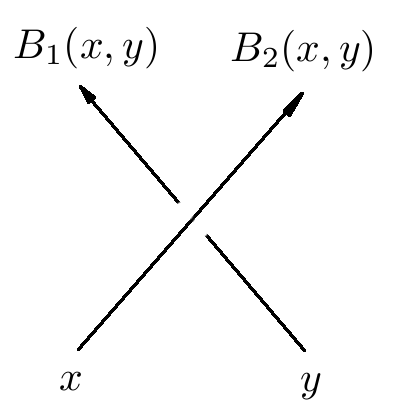}\quad \includegraphics{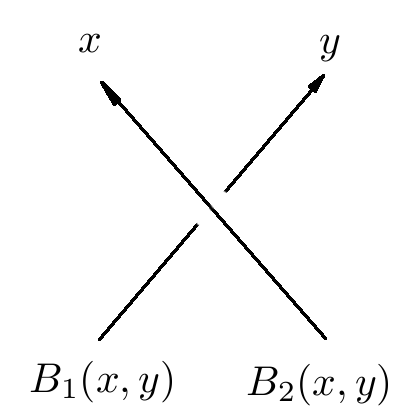}\]
corresponds to a unique such labeling after applying 
any of the blackboard-framed Reidemeister moves. Thus, the number of such
labelings is an easily computable invariant of blackboard-framed isotopy.

The condition that $B$ is a solution to the set-theoretic Yang-Baxter equation
is equivalent to the condition that labelings are preserved by the type III 
move with Cartesian product $\times$ indicating horizontal stacking and 
composition $\circ$ indicating vertical stacking.

\[\includegraphics{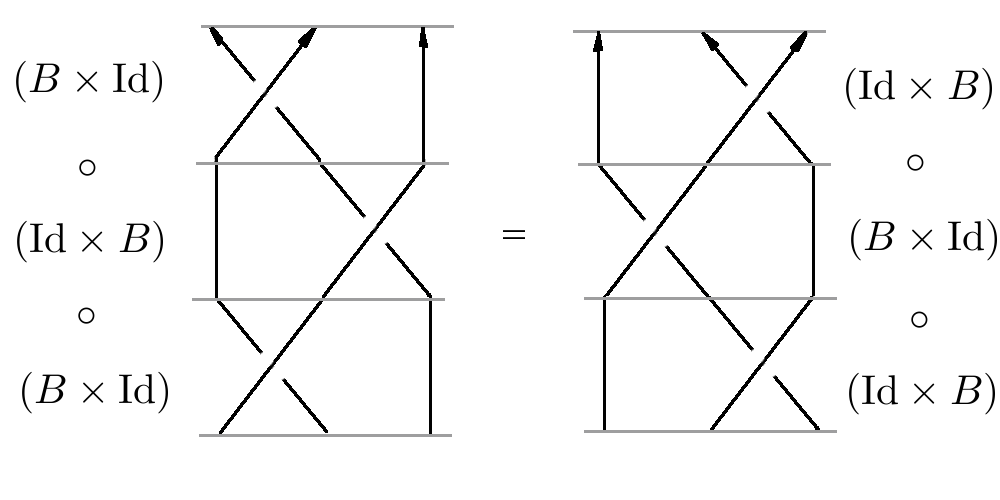}\]

Interpreting $B^{-1}$ as the map defined at a negative crossing satisfies
the \textit{direct} Reidemeister II moves, i.e. the two two-strand moves
where both strands are oriented in the same direction.

\[\includegraphics{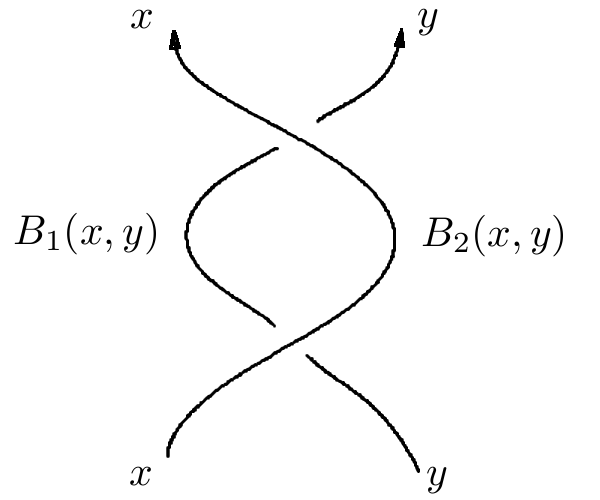}\quad \includegraphics{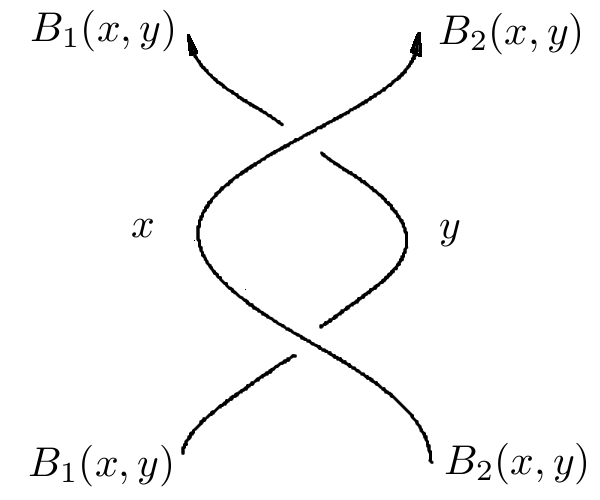}\]

The sideways invertibility requirement is needed to guarantee the existence 
and uniqueness of the labels for the middle semiarcs in the \textit{reverse} 
type II moves, where the strands are oriented in opposite directions:
\[\includegraphics{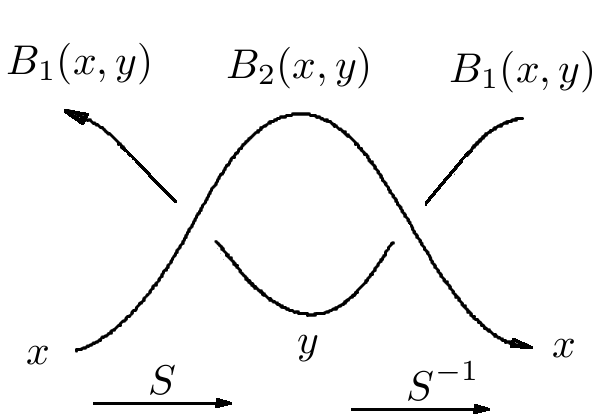}\quad \includegraphics{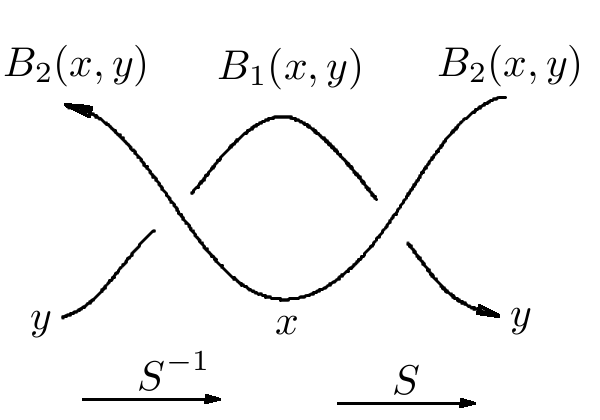}\]

\begin{remark}\textup{
The existence, uniqueness and invertibility of the sideways map $S$ are
equivalent to the
requirement that the component maps $B_1$ and $B_2$ of $B$ are left-
and right-invertible respectively; this condition is sometimes called
the \textit{strong birack} condition. Not imposing this requirement allows
single labelings before a crossing-introducing reverse type II to
branch into multiple labelings after the move, a situation we must avoid
if we want well-defined counting invariants.}
\end{remark}

\begin{remark}\textup{
While there are eight total oriented Reidemeister III moves, the other seven 
moves follow from the pictured III move and direct and reverse II moves.
Thus, the various identities required by the other type III moves are 
satisfied by a birack.}
\end{remark}

The diagonal invertibility condition is required in order to guarantee 
the existence and uniqueness of the labels in the blackboard-framed type I
moves. Of particular importance are the bijections 
$\alpha=(S_2^{-1}\circ\Delta)^{-1}$
and $\pi=S_1^{-1}\circ\Delta\circ \alpha$; these give the labels on a strand
after a framed type I move as pictured.

\[\raisebox{-0.7in}{\includegraphics{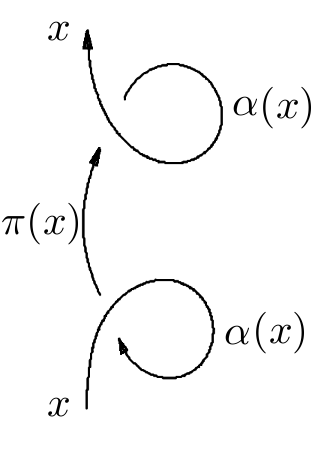}} \quad 
\begin{array}{rcl}
S(\pi(x),x)=(\alpha(x),\alpha(x))
\end{array}\]

In fact, \textit{a priori} we have two potentially distinct bijections 
$\pi$ and $\phi$ coming from the two oriented double-I moves. It turns 
out that the birack axioms imply that the two kink maps coincide:

\begin{proposition}
\textup{Let $(X,B)$ be a birack and $\phi,\pi:X\to X$ be the maps 
$\pi=S_1^{-1}\circ\Delta\circ (S_2^{-1}\circ\Delta)^{-1}$ and 
$\phi=S_1\circ\Delta\circ(S_2\circ\Delta)^{-1}$
as pictured.
Then $\phi=\pi$.}
\[\includegraphics{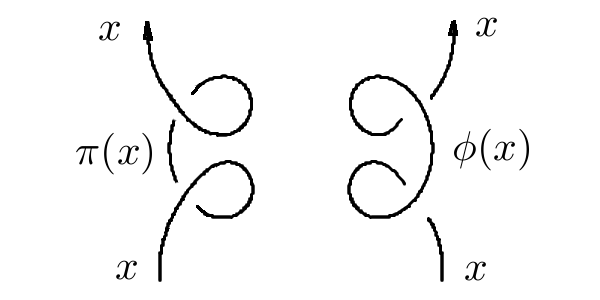}\]
\end{proposition}

\begin{proof}
Since $(X,B)$ is birack, $B$ satisfies the set-theoretic Yang-Baxter 
equation and $S$ satisfies the invertibility requirements for type II moves.
Thus, labelings of knot diagrams by elements of $X$ satisfying the
birack labeling condition are preserved under regular isotopy. Then 
the fact that (see \cite{FR} for instance) 
opposite-writhe opposite-winding-number twists can be canceled using only 
type II and type III moves implies that $\phi^{-1}(\pi(x))=x$ for all $x\in X$,
and we have $\phi=\pi$. 
\[ \includegraphics{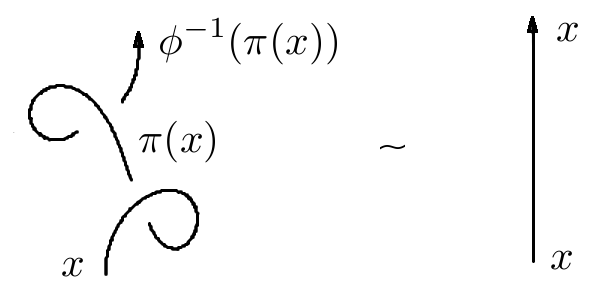}\]
\end{proof}

\begin{definition}
\textup{Let (X,B) be a birack. Then the bijection $\pi:X\to X$ 
given by $\pi=S_1^{-1}\circ\Delta\circ (S_2^{-1}\circ\Delta)^{-1}$
is the \textit{kink map} of $(X,B)$.
$\pi(x)$ represents  ``going through a positive kink'' while 
$\pi^{-1}$ represents ``going through a negative 
kink''.}
\end{definition}

Next we have a few basic definitions relating to the algebra of 
biracks.

\begin{definition}
\textup{Let $Y$ be a subset of $X$. Then $Y$ is a \textit{subbirack} 
of $(X,B)$ if the images of the restrictions of the components of $B$ and 
$S$ to $Y\times Y$ are contained in $Y$.}
\end{definition}

\begin{definition}
\textup{Let $(X,B)$ and $(X',B')$ be biracks. Then a map
$h:X\to X'$ is a \textit{birack homomorphism} if for all 
$x,y\in X$ we have}
\[h(B_1(x,y))=B_1'(h(x),h(y)) \quad \mathrm{and}\quad 
h(B_2(x,y))=B_2'(h(x),h(y)).\]
\textup{For any birack homomorphism $h:X\to X'$ the \textit{image} of $h$,
denoted $\mathrm{Im}(h)$, is the subbirack of $X'$ generated by the elements
$h(x)$ for $x\in X$.}
\end{definition}

\begin{definition}
\textup{Let $(X,B)$ be a birack and let $\pi:X\to X$ be the
kink map. The \textit{birack rank} of $(X,B)$, denoted $N(X,B)$ or just 
$N$ when $(X,B)$ is clear from context, is the smallest positive integer
$N$ such that $\pi^N(x)=x$ for all $x\in X$. In particular, if $X$ is a 
finite set, then $N$ is the exponent of $\pi$
considered as an element of the symmetric group $S_X$.}
\end{definition}

If $N(X,B)<\infty$ then there is a bijection between the sets of labelings
of any two oriented blackboard-framed link diagrams $D$ and $D'$ which are 
related by blackboard framed moves together with the
\textit{$N$-phone cord move}:
\[\includegraphics{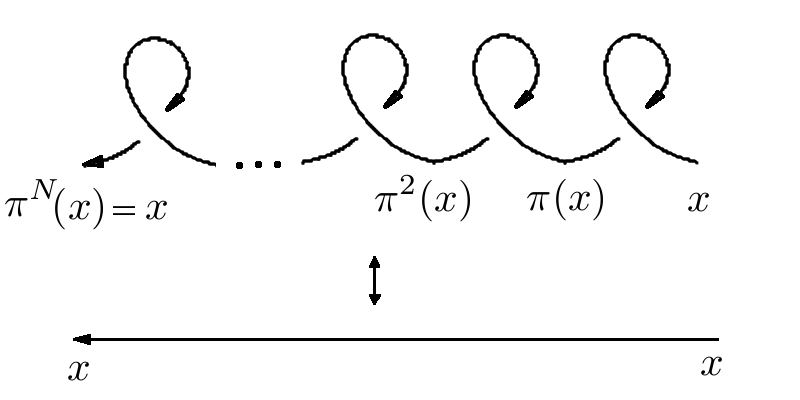}\]

\begin{example}
\textup{Let $X$ be any set, and let $\rho,\tau:X\to X$ be bijections.
Then $B(x,y)=(\tau(y),\rho(x))$ defines a birack provided
$\rho$ and $\tau$ commute. The reader is invited to verify that the
Yang-Baxter equation components become}
\[ \rho^2(x)=\rho^2(x),\quad \tau \rho(y)=\rho\tau(y),\quad 
\mathrm{and} \quad \tau^2(z)=\tau^2(z).\]
\textup{Moreover, we have $B^{-1}(x,y)=(\rho^{-1}(y),\tau^{-1}(x))$,
$S(u,v)=(\rho(v),\tau^{-1}(u))$ with $S^{-1}(u,v)=(\tau(v),\rho^{-1}(u))$,
$f(x)=\tau^{-1}(x)$ and $g(x)=\rho(x)$. We call a 
birack of this type a \textit{constant action birack}. The kink 
map is $\pi(x)=\rho\tau^{-1}(x)$, so $N(X,B)$ is the order of the permutation 
$\rho\tau^{-1}$ in $S_{|X|}$. In particular, a constant action birack is a 
strong biquandle iff $\rho=\tau$.}
\end{example}

\begin{example}
\textup{Any oriented blackboard-framed link diagram 
$L=L_1\cup \dots \cup L_c$ has a \textit{fundamental birack} 
denoted $BR(L)$ (or $BR(L,(w_1,\dots, w_c))$ if we wish to explicitly 
specify the writhe numbers of the components). Let $G=\{g_1,\dots, g_n\}$ 
correspond to the semiarcs of $L$. Define the set $BW(G)$ of 
\textit{birack words} recursively by $x\in BW(G)$ iff $x\in G$, 
$x=B^{\pm 1}_{i}(y,z)$, $x=S^{\pm 1}_{i}(y,z)$, $x=f^{\pm 1}(y)$ or 
$x=g^{\pm 1}(y)$ where $y,z\in BW(G)$ and $i\in\{1,2\}$. Then the
\textit{free birack} on $G$ is the set of equivalence classes
of $BW(G)$ under the equivalence relation generated by the various equations
coming from the Yang-Baxter equation and the strong invertibility requirements
(e.g., $g(x)\sim S_1(x,x)$, etc.). The fundamental birack $BR(L)$
of the link $L$ is then the set of equivalence classes of free blackboard 
birack elements under the additional equivalence relation 
generated by the crossing relations in $L$. We will generally express the 
birack of $G$ by $BR(L)=\langle G \ |\ R\rangle$ where $G$ is
the set of arc labels and $R$ is the set of crossing relations, e.g.}
\[\raisebox{-0.5in}{\includegraphics{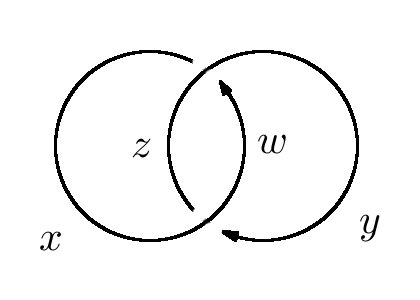}} \quad 
BR(L)=\langle x,y,z,w \ |\ B(x,y)=(z,w),
B(z,w)=(x,y)\rangle.\]
\textup{The quotient of $BR(L)$ obtained by setting $B_2(x,y)$ equal to 
$x$ in all relations is the \textit{fundamental rack} $FR(L)$ of the
blackboard-framed link $L$; the quotient of $BR(L)$ obtained by
setting $\pi(x)=x$ for all $x\in BR(L)$ is the \textit{fundamental (strong)
biquandle} of $L$. Imposing both conditions yields the \textit{knot quandle}.
In particular, since $BR(L)$ determines the fundamental rack of $L$, the 
isomorphism class of $BR(L)$ is a \textit{complete} invariant in the same 
sense as the fundamental rack and the knot quandle; see \cite{FR,J,M}.}
\end{example}

\begin{remark}
\textup{The cardinality of the set of labelings of a blackboard-framed link 
diagram $L$ by a birack $(X,B)$ satisfying the labeling condition 
above is an invariant of blackboard framed isotopy we call the \textit{basic 
counting invariant}, denoted 
\[|\mathrm{Hom}(BR(L),(X,B))|=|\{h\ |\ h:BR(L)\to X \ \mathrm{birack 
\ homomorphism}\}|.\] 
Indeed, each labeling satisfying the crossing condition
determines a unique homomorphism $f:BR(L)\to X$ and every birack
homomorphism corresponds to a unique labeling of $L$.}
\end{remark}

\begin{example}
\textup{Many previously studied algebraic structures are special cases of
biracks.}
\begin{list}{$\bullet$}{}
\item \textup{A birack with $\pi=\mathrm{Id}:X\to X$ is a 
\textit{(strong) biquandle} (see \cite{KR,NV,FJK});}
\item \textup{A birack with $\pi=\mathrm{Id}$ and $B^{-1}=B$
is a \textit{semiquandle} (see \cite{HN});}
\item \textup{A birack with $B_2(x,y)=x$ is a \textit{rack}
(see \cite{FR,N});}
\item \textup{A birack with both $B_2(x,y)=x$ and $\pi=\mathrm{Id}$
is a \textit{quandle} (see \cite{J}).}
\end{list}
\end{example}

To define birack structures on a finite set $X=\{x_1,\dots, x_n\}$ 
without algebraic formulas for the component maps $B_i$, we can give a pair of
$n\times n$ matrices with entries in $\{1,2\dots, n\}$ encoding the component
maps of $B$. Specifically, the \textit{birack matrix} of a birack
$X$ is the block matrix $M_B=[B_1 | B_2]$ such that the entries in row $i$ 
column $j$ of $B_1$ and $B_2$ respectively are $x_k$ and $x_h$ where 
$B_1(x_i,x_j)=x_k$ and $B_2(x_j,x_i)=x_h$. Note the transposition of the
input variables for $B_2$; this is for backwards compatibility with notation 
from previous work. Comparing the notation from
\cite{NV}, our $B_1$ is the upper right block matrix and our $B_2$ is the 
lower right block matrix. 

\begin{example}\label{cab1}
\textup{As an easy example, let $X=\{1,2,3,4\}$ and let $\tau=(12)$ and
$\rho=(34)$. Then $\rho\tau=\tau\rho$ and we have a birack
with matrix}
\[M_B=\left[\begin{array}{cccc|cccc}
2 & 2 & 2 & 2 & 1 & 1 & 1 & 1 \\ 
1 & 1 & 1 & 1 & 2 & 2 & 2 & 2 \\
3 & 3 & 3 & 3 & 4 & 4 & 4 & 4 \\
4 & 4 & 4 & 4 & 3 & 3 & 3 & 3 \\
\end{array}\right].\]
\textup{The kink map here is $\pi=(12)(34)$, and thus we have 
birack rank $N=2$.}
\end{example}

\section{\large\textbf{$(t,s,r)$-biracks and $(\tau,\sigma,\rho)$-biracks}}
\label{S3}

In this section we define two classes of biracks,
$(t,s,r)$-biracks and $(\tau,\sigma,\rho)$-biracks. We begin
with the simpler case. Let $\tilde\Lambda=\mathbb{Z}[t^{\pm 1}, s,r^{\pm 1}]/I$ 
where $I$ is the ideal generated by $s^2-(1-tr)s$. Let $X$ be any $\tilde\Lambda$ module.

\begin{lemma}
The element $\pi=tr+s\in\tilde\Lambda$ is invertible with inverse given by
$\pi^{-1}=t^{-1}r^{-1}(1-s)$.
\end{lemma}

\begin{proof}
\begin{eqnarray*}
(tr+s)t^{-1}r^{-1}(1-s) & = & (tr+s)(t^{-1}r^{-1}-t^{-1}r^{-1}s) \\
 & = & 1+t^{-1}r^{-1}s-s-t^{-1}r^{-1}s^2 \\
 & = & 1+t^{-1}r^{-1}s-s-t^{-1}r^{-1}(1-tr)s \\
 & = & 1+(t^{-1}r^{-1}-1)s-t^{-1}r^{-1}(1-tr)s \\
 & = & 1+(t^{-1}r^{-1}-1)s-(t^{-1}r^{-1}-1)s \\
 & = & 1.
\end{eqnarray*}
\end{proof}

\begin{corollary}
The element $(1-s)\in \tilde\Lambda$ is invertible with inverse
$(1-s)^{-1}=1+t^{-1}r^{-1}s$.
\end{corollary}

\begin{proposition}
Let $X$ be a $\tilde\Lambda$-module and define $B(x,y)=(ty+sx,rx)$.
Then $(X,B)$ is a birack with kink map $\pi(x)=(tr+s)x$.
\end{proposition}

\begin{proof}

We must check that $B$ is strongly invertible. First, note that
$B^{-1}(x,y)=(r^{-1}x,t^{-1}y-st^{-1}r^{-1}x)$ so $B$ is invertible. Next, 
the sideways map $S$ is given by $S(x,y)=(ry,t^{-1}x-t^{-1}sy)$ with 
inverse $S^{-1}(x,y)=(ty+sr^{-1}x,r^{-1}x)$. Finally, for diagonal 
invertibility, we have $f(x)=t^{-1}(1-s)x$ and $g(x)=r(x),$ so
$\pi(x)=(g\circ f^{-1})(x)=t(1-s)^{-1}rx=tr(1+t^{-1}r^{-1}s)x=(tr+s)x$.

Next, we must check that the set-theoretic Yang-Baxter equation is satisfied.
In component form, we have
\begin{eqnarray*}
t^2z+tsy+sx & = & t^2z+sty+(str+s^2)x \\
r(ty+sx) & = & t(ry)+s(rx) \\
r^2 x & = & r^2 x.
\end{eqnarray*}
Commutativity in $\tilde\Lambda$ satisfies the second equation and reduces the 
first equation to $(1-tr)s=s^2$.
\end{proof}

For the purpose of defining counting invariants with $(t,s,r)$-biracks,
we'll need \textit{finite} examples. The simplest way to do this is to
choose a finite abelian group, say $A=\mathbb{Z}_n$, and choose elements
$t,s,r\in A$ such that $t,r$ are invertible and $s^2=(1-tr)s$. Then the
set $A^m$ of $m$-tuples of $A$ forms a finite $(t,s,r)$-rack.

\begin{corollary}
A finite $(t,s,r)$-birack structure $(A,B)$ on an abelian group
$A$ is a strong biquandle iff $tr+s=1$.
\end{corollary}

\begin{corollary}
The birack rank of a finite $(t,s,r)$-birack is the smallest integer $N>0$
such that $(tr+s)^N=1$.
\end{corollary}

\begin{example}
\textup{Let $A=\mathbb{Z}_4$ and set $t=3,s=2$ and $r=3$. Then 
$s^2=4=0=2(1-9)=s(1-tr)$ so $A^m$ is a finite $(t,s,r)$-birack 
with birack operation $B(x,y)=(3y+2x,3x)$.
$|A^m|=4^m$ and $A_m$ has birack rank $2$ since $tr+s=9+2=3$ and $3^2=9=1$.}
\end{example}

\begin{remark}\textup{
$(t,s,r)$-biracks have several important special cases. If we set
$r=1$, we have a \textit{$(t,s)$-rack} as defined in \cite{FR}.
introduced in \cite{FR} and subsequently studied in papers such as \cite{CN2}. 
In the case $s=1-tr$
we have a biquandle known as an \textit{Alexander biquandle}, introduced in
\cite{KR} and subsequently studied in papers such as \cite{LN}. If we
set $r=1$ and $s=1-t$, we have an \textit{Alexander quandle}, introduced
in \cite{J} and studied in many papers such as \cite{AG, N3}.}
\end{remark}

\begin{remark}\textup{
The general case of biquandle structures defined via linear operations
on modules over non-commutative rings has been
studied in several recent papers \cite{BF1, BF2, FT}. Certain cases of these 
(known as \textit{quaternionic biquandles}) yield invariants of
virtual knots analogous to the Alexander polynomial.
}\end{remark}

Every $\tilde\Lambda$-module is an abelian group under addition with 
multiplication by $t$ and $r$ acting as automorphisms of the additive 
structure and multiplication by $s$ acting as an endomorphism. This
suggests a way to generalize the $(t,s,r)$-birack definition by 
unabelianizing the group structure.

\begin{proposition}
\textup{Let $G$ be a group with $\tau,\rho\in \mathrm{Aut}(G)$ and
$\sigma\in \mathrm{End}(G)$ such that $\rho$ commutes with $\tau$ and 
$\sigma$ and satisfying for all $y,z\in G$}
\begin{equation}\label{Eq1}
\tau\sigma(y)\sigma(z)=\tau\sigma\rho(z)\sigma\tau(y)\sigma^2(z). 
\end{equation}
\textup{Define $B:G\times G\to G\times G$ by 
$B(x,y)=(\tau(y)\sigma(x),\rho(x))$. Then $(G,B)$ is a birack 
with kink map given by $\pi(x)=\tau\rho(x)\sigma(x)$; we call this
a \textit{$(\tau,\sigma,\rho)$-birack.}}
\end{proposition}

\begin{proof}
It is a straightforward computation to show that 
\begin{eqnarray*}
B^{-1}(x,y) & = & (\rho^{-1}(y),\tau^{-1}(x(\sigma\rho^{-1}(y))^{-1})),\\
S(x,y) & = & (\rho(y),\tau^{-1}(x\sigma(y)^{-1})), \\ 
S^{-1}(x,y) & = & (\tau(y)\sigma(\rho^{-1}(x)),\rho^{-1}(y)) \\ 
\end{eqnarray*}
and that we have
diagonal invertibility with kink map given by $\pi(x)=\tau\rho(x)\sigma(x)$.

To see that the set-theoretic Yang-Baxter equation is satisfied,
we note that $B$ gives us the Yang-Baxter component form equations
\begin{eqnarray*}
\tau^2(x)\tau\sigma(y)\sigma(z) & = & 
\tau^2(x)\tau\sigma\rho(z)\sigma\tau(y)\sigma^2(z) \\
\tau\rho(y)\sigma\rho(z) & = & \rho\tau(y)\rho\sigma(z) \\
\rho^2(z) & = & \rho^2(z). 
\end{eqnarray*}
The first equation reduces to equation (\ref{Eq1}) and the second requires
that $\rho$ commute with $\sigma$ and $\tau$.
\end{proof}

A few special cases of $(\tau,\sigma,\rho)$-biracks appear in the literature;
the case where the kink map $\pi(x)=\mathrm{Id}(x)$ is a biquandle known
as a \textit{Silver-Williams switch} \cite{FJK}, and in the quandle
case, i.e. $\rho(x)=\mathrm{Id}(x)$ and $\sigma(x)=\tau(x^{-1})x$, we have
a quandle structure incorrectly referred to by the present author in previous
work as ``homogeneous quandles''; a better term would be \textit{multiplicative
Alexander quandles} as these need not be homogeneous in the sense of \cite{J}.
The special case where $G$ is the automorphism group 
$\mathrm{Aut}(Q)$ of a quandle $Q$ and $\tau$ is conjugation in $G$ by 
a chosen inner automorphism turns out to be the key to the relationship 
between the knot group and the knot quandle; see \cite{J,M} for more.

\begin{example}
\textup{Let $G=\langle \alpha,\beta\ |\ \alpha^4=1,\ \beta^m=1, \
\alpha\beta=\beta\alpha^{-1} \rangle$ and let 
$\tau(\alpha^i\beta^j)=\rho(\alpha^i\beta^j)=\beta^j\alpha^i$ and
$\sigma(\alpha^i\beta^j)=\alpha^{2i}$. Then $\tau,\rho\in \mathrm{Aut}(G)$,
$\sigma\in\mathrm{End}(G)$ and we have
$\rho\tau=\mathrm{Id}=\tau\rho \quad \mathrm{and}\quad 
\rho\sigma=\sigma=\sigma\rho.$ Then for  $y=\alpha^i\beta^j$ and $z=\alpha^k\beta^l$, we have}
\[\tau\sigma(y)\sigma(z)=
\alpha^{2i}\alpha^{2k}=\alpha^{2k}\alpha^{2i}\alpha^{4k}=
\tau\sigma\rho(z)\sigma\tau(y)\sigma^2(z)\]
\textup{and thus we have a $(\tau,\sigma,\rho)$-birack. Here the kink map
is given by $\pi(\alpha^i\beta^j)=\alpha^{3i}\beta^j,$ so we have birack rank
$N=2$.}

\end{example}

\section{\large\textbf{Birack link invariants}} \label{S4}

By construction, there is a bijection between the sets of labelings of any
two blackboard-framing-isotopic link diagrams by the same birack
$(X,B)$. In particular, as we increment the writhe of the components in
the diagram, the number of labelings of a knot diagram by a birack 
has period $N$ in each component where $N$ is the birack rank. As in \cite{N}, 
we can sum these basic counting invariants over a complete period of framings 
to define an invariant of unframed ambient isotopy classes of oriented links,
provided the labeling birack $X$ is finite. In particular, a link 
$L=L_1\cup \dots \cup L_c$ of $c$ components has writhe vectors in
$W=(\mathbb{Z}_N)^c$ with respect to the labelings by a blackboard
birack $(X,B)$.

\begin{definition}
\textup{Let $(X,B)$ be a birack with birack rank $N$, 
$L=L_1\cup\dots\cup L_c$ a $c$-component link and $W=(\mathbb{Z}_N)^c$. Then
the \textit{integral birack counting invariant} of $L$ with 
respect to $(X,B)$ is}
\[\phi_{(X,B)}^{\mathbb{Z}}(L;(X,B))=\sum_{\mathbf{w}\in W} 
|\mathrm{Hom}(BR(L,\mathbf{w}),X)|.\]
\end{definition}

\begin{remark}
\textup{If we define blackboard-framed isotopy of virtual links in the
obvious way, namely as the result of replacing the classical type I moves
with blackboard-framed double type I moves and keeping all other moves
the same, then ignoring virtual crossings lets us extend the counting 
invariant and its enhancements below to virtual knots without modification.}
\end{remark}

\begin{example} \label{ex:hopf}
\textup{The smallest birack which is neither a biquandle nor a
rack is the two-element constant action birack $X=\{1,2\}$ with 
$\tau=\mathrm{Id}_X$ and $\rho=(12)$, i.e., with birack matrix}
\[M_B=\left[\begin{array}{cc|cc}
1 & 1 & 2 & 2 \\
2 & 2 & 1 & 1
\end{array}\right].
\]
\textup{Interpreted as labeling rules, $(X,B)$ says we switch from 1 to 2 or
2 to 1 when going over a crossing and keep the same label when going under
a crossing. In particular, the kink map $\pi$ is the transposition $(12)$
so we have birack rank $N=2$. The counting invariant $\phi^{\mathbb{Z}}_{BR}(L)$
of a link $L$ is then the total number of colorings over a 
complete set of diagrams of $L$ with every combination of 
even and odd writhes on each component. Thus, both the Hopf link and the
unlink of two components have four labelings by $(X,B)$}

\[\begin{array}{|c|c|c|c|} \hline
\mathbf{w}=(0,0) & \mathbf{w}=(1,0) & 
\mathbf{w}=(0,1) & \mathbf{w}=(1,1) \\ \hline
\scalebox{0.7}{\includegraphics{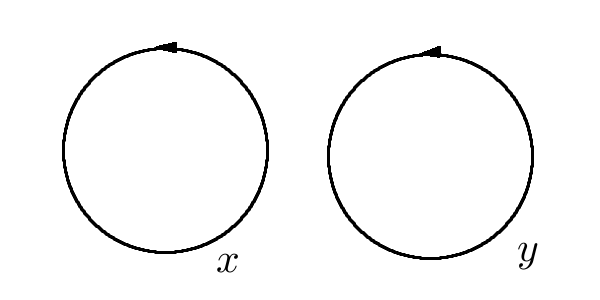}} & 
\scalebox{0.7}{\includegraphics{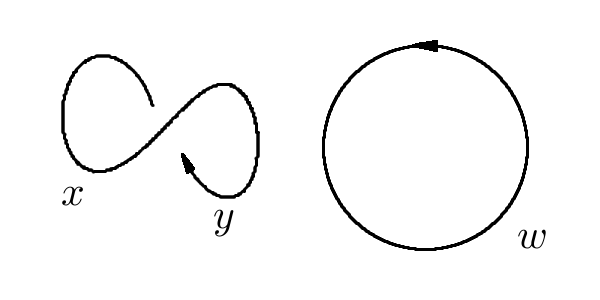}} &
\scalebox{0.7}{\includegraphics{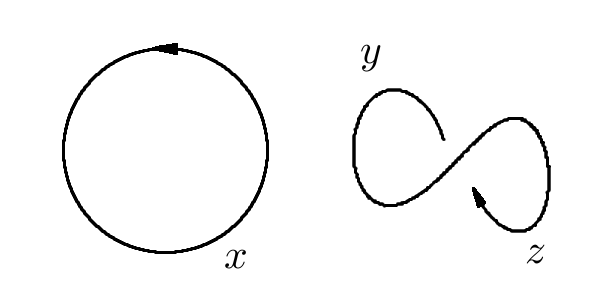}} & 
\scalebox{0.7}{\includegraphics{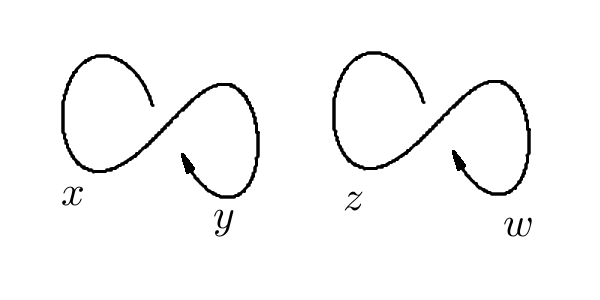}} \\ \hline
\begin{array}{cc}
x & y \\ \hline
1 & 1 \\
1 & 2 \\
2 & 1 \\ 
2 & 2 \\
\end{array} &
\begin{array}{ccc}
x & y & w \\ \hline 
- & - & - \\
\end{array} &
\begin{array}{ccc}
x & y & z \\ \hline
- & - & - \\
\end{array} &
\begin{array}{cccc}
x & y & z & w \\ \hline
- & - & - & - \\
\end{array} \\ \hline
\end{array}
\]
\[\begin{array}{|c|c|c|c|} \hline
\mathbf{w}=(0,0) & \mathbf{w}=(1,0) & 
\mathbf{w}=(0,1) & \mathbf{w}=(1,1) \\ \hline
\scalebox{0.7}{\includegraphics{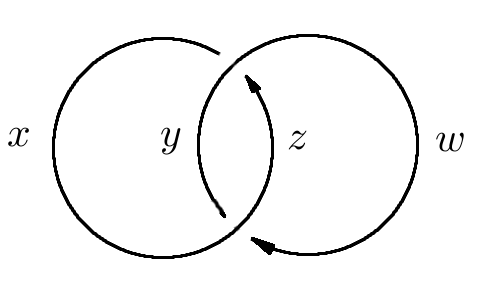}} & 
\scalebox{0.7}{\includegraphics{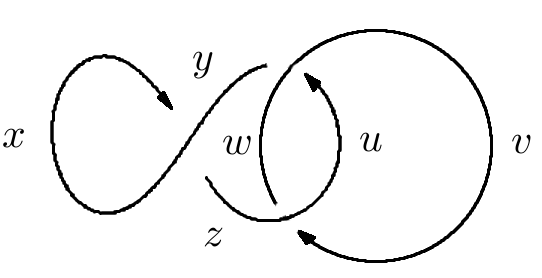}} &
\scalebox{0.7}{\includegraphics{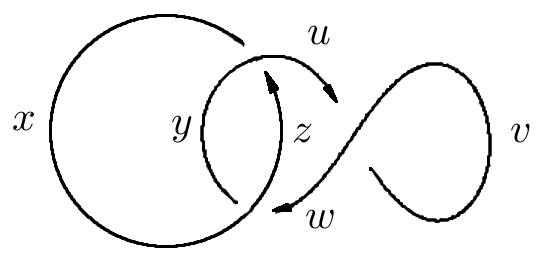}} & 
\scalebox{0.7}{\includegraphics{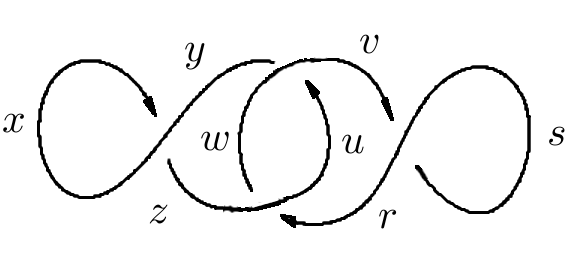}} \\ \hline
\begin{array}{cccc}
x & y & z & w\\ \hline
- & - & - & - \\
\end{array} &
\begin{array}{cccccc}
x & y & z & w & u & v \\ \hline
- & - & - & - & - & - \\
\end{array} &
\begin{array}{cccccc}
x & y & z & w & u & v \\ \hline
- & - & - & - & - & - \\
\end{array} &
\begin{array}{cccccccc}
x & y & z & w & u & v & r & s \\ \hline
1 & 2 & 1 & 1 & 2 & 2 & 1 & 2 \\
1 & 2 & 1 & 2 & 2 & 1 & 2 & 1 \\
2 & 1 & 2 & 1 & 1 & 2 & 1 & 2 \\
2 & 1 & 2 & 2 & 1 & 1 & 2 & 1 \\  
\end{array} \\ \hline
\end{array}
\]

\end{example}

We notice that the total numbers of labelings of both links
over one framing period mod $N$ are the same, but the labelings occur
in different framings. To exploit this information, we will need to
enhance the counting invariant.

An \textit{enhancement} of a counting invariant is an invariant which 
specializes to the counting invariant but includes extra information
to distinguish between labelings, making enhanced invariants stronger
than the original counting invariants. One way to enhance counting
invariants is to define a ``signature'' for each labeling which is 
well-defined under Reidemeister moves on the labeled diagram; the 
resulting multiset of signatures is then an enhancement of the
counting invariant. Examples include using quandle 2-cocycles to
define a \textit{Boltzmann weight} signature for quandle labelings
of $L$ (see \cite{CJKLS}) as well as using extra structure of the
labeling object to define signatures (see \cite{NN, NW} etc.).
Taking a generating function lets us express 
these multiset-valued invariants as polynomials by converting 
multiplicities to coefficients and signatures to exponents, e.g.
the multiset $\{0,0,1,1,1,2\}$ becomes the polynomial $2+3z+z^2$.

As with the rack counting invariant, one easy type of enhancement involves 
keeping track of which framings contribute which labelings.  
Let $q^{\mathbf{w}}=\prod_{i=1}^c q_i^{w_i}$ for 
$\mathbf{w}=(w_1,\dots,w_c)\in W$. Then we have:

\begin{definition}
\textup{Let $(X,B)$ be a birack with birack rank $N$, 
$L=L_1\cup\dots\cup L_c$ a $c$-component link and $W=(\mathbb{Z}_N)^c$.
The \textit{writhe-enhanced birack counting invariant} of $L$ with 
respect to $(X,B)$ is}
\[\phi_{(X,B)}^{W,M}(L)=
\{ (|\mathrm{Hom}(BR(L,\mathbf{w}),(X,B))|,\mathbf{w})\ :\ \mathbf{w}\in W\}\]
\textup{in multiset form; in polynomial form we have:}
\[\phi_{(X,B)}^{W}(L)=\sum_{\mathbf{w}\in W} 
|\mathrm{Hom}(BR(L,\mathbf{w}),X)|q^{\mathbf{w}}.\]
\end{definition}

\begin{example}
\textup{The Hopf link and two-component unlink, while not distinguished by
the integral counting invariant determined by the birack in example
\ref{ex:hopf}, are distinguished by the writhe-enhanced invariant determined
by the same birack, with $\phi_{(X,B)}^{W}(\mathrm{Hopf})=4q_1q_2\ne
4=\phi_{(X,B)}^{W}(\mathrm{Unlink})$.}
\end{example}

\begin{remark}
\textup{In the case of birack rank $N=1$, i.e. if $(X,B)$ is a quandle or 
strong biquandle, the integral and writhe-enhanced polynomial invariants 
are the same; indeed, in this case both consist of single basic counting 
invariant $|\mathrm{Hom}(BR(L),(X,B))|$. If $B_2(x,y)=x$ so that $(X,B)$ 
is a rack, then the integral and writhe-enhanced
polynomial birack counting invariants are the integral and polynomial
rack counting invariants described in \cite{N}.}
\end{remark}

Another straightforward enhancement uses the cardinality of the image 
subbirack as a signature:
\begin{definition}
\textup{Let $(X,B)$ be a birack with birack rank $N$, 
$L=L_1\cup\dots\cup L_c$ a $c$-component link and $W=(\mathbb{Z}_N)^c$. 
The \textit{image-enhanced birack counting invariant} of $L$ with 
respect to $(X,B)$ is}
\[\phi_{(X,B)}^{\mathrm{Im},M}(L)=\{ |\mathrm{Im}(f)|\ : \ 
f\in\mathrm{Hom}(BR(L,\mathbf{w}),(X,B)), \mathbf{w}\in W\}\]
\textup{in multiset form; in polynomial form we have:}
\[\phi_{(X,B)}^{\mathrm{Im}}(L)=\sum_{\mathbf{w}\in W} \left(
\sum_{f\in\mathrm{Hom}(BR(L,\mathbf{w}),(X,B))} z^{|\mathrm{Im}(f)|}\right).\]
\end{definition}

\begin{example}
\textup{Let $L$ be the trefoil knot $3_1$ and $(X,B)$ the $(t,s,r)$-birack
structure on $X=\mathbb{Z}_3$ given by $B(x,y)=(y+2x,2x)$. Here the kink map
is $\pi(x)=(tr+s)x=(2+2)x=x$ so we have $N=1$ and can compute the invariant
from a diagram with any writhe. Since $\mathbb{Z}_3$ is a field,
we can find the $\mathbb{Z}_3$-vector space of labelings of $L$ using linear
algebra:}
\[\includegraphics{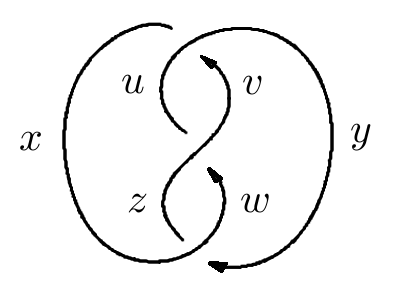}\]
\textup{We have a presentation 
$BR(3_1)=\langle x,y,z,w,u,v\ |\ 
B(x,y) =(z,w), B(z,w)=(u,v), B(u,v)=(x,y) \rangle$. This translates into 
a coefficient matrix for a homogeneous system of linear equations, which 
we row-reduce over $\mathbb{Z}_3$:}
\[
\begin{array}{rcl}
2x+y+2z & = & 0 \\
2x + 2w & = & 0 \\
2z+w+2u & = & 0 \\
2z+ 2v  & = & 0 \\
2x+2u+v & = & 0 \\
2y+ 2u  & = & 0
\end{array}\quad \to \quad
\left[\begin{array}{cccccc}
2 & 1 & 2 & 0 & 0 & 0 \\
2 & 0 & 0 & 2 & 0 & 0 \\
0 & 0 & 2 & 1 & 2 & 0 \\
0 & 0 & 2 & 0 & 0 & 2 \\
2 & 0 & 0 & 0 & 2 & 1 \\
0 & 2 & 0 & 0 & 2 & 0 \\
\end{array}\right] \quad \to \quad
\left[\begin{array}{cccccc}
1 & 0 & 0 & 0 & 1 & 2 \\
0 & 1 & 0 & 0 & 1 & 0 \\
0 & 0 & 1 & 0 & 0 & 1 \\
0 & 0 & 0 & 1 & 2 & 1 \\
0 & 0 & 0 & 0 & 0 & 0 \\
0 & 0 & 0 & 0 & 0 & 0 \\
\end{array}\right]
\]
\textup{Thus we can see that the set of labelings is the $\mathbb{Z}_3$-span of
the set $\{(1,1,0,2,2,0),(2,0,1,1,0,2)\}$; thus, eight of the $3^2=9$ labelings
considered  as birack homomorphisms $f:BR(3_1)\to (\mathbb{Z}_3,B)$ are 
surjective, while the trivial linear combination represents the zero
homomorphism. Thus, the integral counting invariant is 
$\phi^{\mathbb{Z}}_{(\mathbb{Z}_3,B)}(3_1)=9$ while the image-enhanced invariant is
$\phi^{\mathrm{Im}}_{(\mathbb{Z}_3,B)}(3_1)=1+8z^3$. Similar computations yield
$\phi^{\mathrm{Im}}_{(\mathbb{Z}_3,B)}(\mathrm{Unknot})=1+2z^3$ and 
$\phi^{\mathrm{Im}}_{(\mathbb{Z}_3,B)}(4_1)=1$.}
\end{example}

\begin{remark}
\textup{As the above example illustrates, the set of labelings of a knot
or link by a $(t,s,r)$-birack structure on $(\mathbb{Z}_p)^n$ for $p$
prime is a $\mathbb{Z}_p$-vector space; hence, the counting invariant
$\phi_{((\mathbb{Z}_p)^n,B)}^{\mathbb{Z}}(L)$ equals $p^m$ for some $m\ge 0$.}
\end{remark}

\begin{remark}
\textup{Say that a birack is \textit{simple} if it has no 
non-empty proper subbiquandles, e.g. the birack in example \ref{ex:hopf}. 
Then for simple  biracks, the integral 
and image-enhanced invariants contain equal
information; more precisely, for simple biracks we have
\[\phi^{\mathbb{Z}}_{(X,B)}(L)z^{|X|}=\phi^{\mathrm{Im}}_{(X,B)}(L).\] Note
that if $(X,B)$ is a quandle then singletons are proper subbiracks, so
quandles with $|X|\ge 2$ are not simple biracks.}
\end{remark}

For our final enhancement of the birack counting invariants,
we will use \textit{birack polynomials}. In \cite{NP} a two-variable 
polynomial invariant is defined for finite quandles and extended to 
biquandles in \cite{NP2} and racks in \cite{CN}. These polynomials quantify 
the way in which the trivial action of elements of $X$
on $X$ is distributed throughout the birack, as opposed to
being concentrated in a single identity element as in a group. 
We will use a simplified version of the biquandle polynomial to define 
an enhancement of the birack counting invariant. 

\begin{definition}
\textup{Let $(X,B)$ be a finite birack. 
For each $x\in X$ let }
\[
c_1(x)  =  |\{y\in X \ : \ B_1(y,x) = y\}|, \quad
c_2(x)  =  |\{y\in X \ : \ B_2(x,y) = y\}|, \]\[
r_1(x)  =  |\{y\in X \ : \ B_1(y,x) = x\}|, \quad \mathrm{and} \quad
r_2(x)  =  |\{y\in X \ : \ B_2(x,y) = x\}|. 
\]
\textup{The \textit{birack polynomial} of $(X,B)$, denoted 
$\rho_{(X,B)}(s_1,t_1,s_2,t_2)$, is the four-variable polynomial}
\[\rho_{(X,B)}(s_1,s_2,t_1,t_2)
=\sum_{x\in X} s_1^{c_1(x)}s_2^{c_2(x)}t_1^{r_1(x)}t_2^{r_2(x)}.\]
\end{definition}

\begin{proposition}
If $(X,B)$ and $(X',B')$ are isomorphic biracks, then
\[\rho_{(X,B)}(s_1,s_2,t_1,t_2)=\rho_{(X',B')}(s_1,t_1,s_2,t_2).\]
\end{proposition}

\begin{proof}
If $\psi:X\to X'$ is an isomorphism, then $c_i(\psi(x))=c_i(x)$ and
$r_i(\psi(x))=r_i(x)$ for $i=1,2$.
\end{proof}

Next, we have a polynomial for subbiracks $Y\subset X$. This polynomial 
contains information not just about the subbirack $Y$ itself but also 
about how $Y$ is embedded in $X$. See \cite{NP} for more.

\begin{definition}
\textup{For any subbirack $Y\subset X$ of $(X,B)$, the subbirack polynomial
of $Y$ is given by}
\[\rho_{Y\subset X}(s_1,s_2,t_1,t_2)
=\sum_{y\in Y} s_1^{c_1(y)}s_2^{c_2(y)}t_1^{r_1(y)}t_2^{r_2(y)}.\]
\end{definition}

\begin{example}\label{ex:un}
\textup{The birack with matrix}
\[M_b=\left[\begin{array}{cccc|cccc}
2 & 2 & 1 & 1 & 2 & 2 & 1 & 1 \\
1 & 1 & 2 & 2 & 1 & 1 & 2 & 2 \\
3 & 4 & 3 & 3 & 4 & 3 & 4 & 4 \\
4 & 3 & 4 & 4 & 3 & 4 & 3 & 3
\end{array}\right]\]
\textup{has birack polynomial $\rho_{(X,B)}=s_1^2t_1^2t_2^2 + s_2^2t_1^2t_2^2 + 
2s_1^4s_2^2t_1^3t_2$. There are two subbiracks, $Y=\{1,2\}$ and $Z=\{3,4\}$
with subbirack polynomials $\rho_{Y\subset X}=s_1^2t_1^2t_2^2 + s_2^2t_1^2t_2^2$
and $\rho_{Z\subset X}=2s_1^4s_2^2t_1^3t_2$.}
\end{example}

\begin{definition}
\textup{Let $(X,B)$ be a finite birack with birack rank $N$,
$L$ a link of $c$ components and $W=(\mathbb{Z}_N)^c$. The 
\textit{birack polynomial enhanced invariant} of $L$ with 
respect to $(X,B)$ is the multiset}
\[\phi^{\rho,M}_{(X,B)}(L)=
\left\{ \rho_{\mathrm{Im}(f)\subset X}\ |\ 
f\in \mathrm{Hom}(BR(L,\mathbf{w}),(X,B)),\ 
\mathbf{w}\in W\right\};\]
\textup{in polynomial form we have}
\[\phi^{\rho}_{(X,B)}(L)=
\sum_{\mathbf{w}\in W} \left(\sum_{f\in \mathrm{Hom}(BR(L,\mathbf{w}),(X,B))} 
z^{\rho_{\mathrm{Im}(f)\subset X}}\right).\]
\end{definition}

\begin{example}
\textup{The polynomial enhanced invariant can distinguish labelings 
with image subbiracks which happen to have the same cardinality but
are not isomorphic or are isomorphic but embedded differently in the
overall birack. The birack in example \ref{ex:un} has birack 
rank 2. The unknot has six labelings by this birack over a period of
writhes $w\in\mathbb{Z}_2$ as pictured, with four labelings occurring in
writhe $w=0$ and two occurring in writhe $w=1$.}
\[\scalebox{0.8}{
\includegraphics{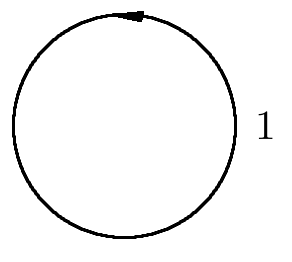}\quad
\includegraphics{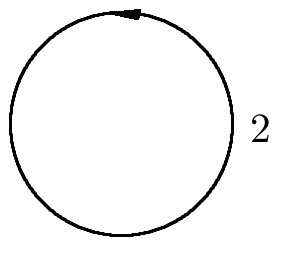}\quad
\includegraphics{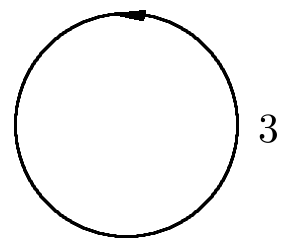}\quad
\includegraphics{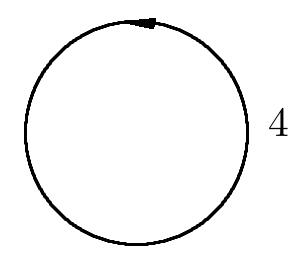} \quad
\includegraphics{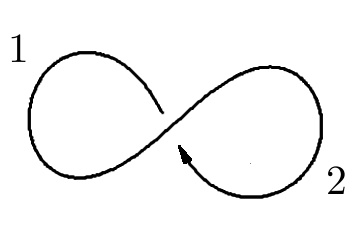} \quad
\includegraphics{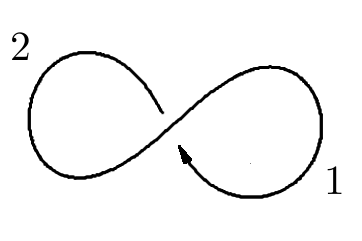} \quad}
\]
\textup{We have}
\[\phi^{\rho}_{(X,B)}(\mathrm{Unknot})
=4z^{s_1^2t_1^2t_2^2 + s_2^2t_1^2t_2^2}+2z^{2s_1^4s_2^2t_1^3t_2}.\]
\end{example}

\begin{remark} \textup{
It is easy to see that the integral counting invariant 
$\phi^{\mathbb{Z}}_{(X,B)}$
can be recovered from $\phi^{\rho}_{(X,B)}(L)$ by specializing $z=1$
and that the image-enhanced invariant can be recovered by
setting $s_i=t_i=1$. When $(X,B)$ is a quandle, $\phi^{\rho}_{(X,B)}(L)$ 
satisfies \[\phi^{\rho}_{(X,B)}(L)=\Phi_{qp}(L)(s_2t_2)^{|X|}\]
where $\Phi_{qp}(L)$ is the quandle polynomial invariant with $s=s_1$ and
$t=t_1$.}
\end{remark}

\begin{example}
\textup{To see that the polynomial-enhanced invariant is stronger than the
counting invariant, we compute (using \texttt{python} code) that the
knots $5_1$ and $6_1$ both have counting invariant $30$ with respect 
to the 10-element birack with matrix listed below, but are distinguished
by the image-enhanced invariant values 
\[\phi^{\rho}_{(X,B)}(5_1)=4z^{s_1^6s_2^{10}t_1t_2^5} + 5z^{s_1^2s_2^6t_1^6t_2^{10}} 
+ z^{s_1^6s_2^{10}t_1^6t_2^{10}} + 20z^{5s_1^2s_2^6t_1^6t_2^{10}}\] and 
\[\phi^{\rho}_{(X,B)}(6_1)=4z^{s_1^6s_2^{10}t_1t_2^5} + 5z^{s_1^2s_2^6t_1^6t_2^{10}} 
+ z^{s_1^6s_2^{10}t_1^6t_2^{10}} + 20z^{4s_1^6s_2^{10}t_1t_2^5 + s_1^6s_2^{10}t_1^6t_2^{10}}
.\]} 
\[
M_B=\left[
\begin{array}{cccccccccc|cccccccccc}
1 & 3 & 5 & 2& 4& 1& 1& 1& 1& 1 & 1& 1& 1& 1& 1& 1& 1& 1& 1& 1 \\
5 & 2 & 4 & 1& 3& 2& 2& 2& 2& 2 & 2& 2& 2& 2& 2& 2& 2& 2& 2& 2 \\
4 & 1 & 3 & 5& 2& 3& 3& 3& 3& 3 & 3& 3& 3& 3& 3& 3& 3& 3& 3& 3 \\
3 & 5 & 2 & 4& 1& 4& 4& 4& 4& 4 & 4& 4& 4& 4& 4& 4& 4& 4& 4& 4 \\
2 & 4 & 1 & 3& 5& 5& 5& 5& 5& 5 & 5& 5& 5& 5& 5& 5& 5& 5& 5& 5 \\
7 & 7 & 7 & 7& 7& 6& 10& 9& 8& 7 & 8& 8& 8& 8& 8& 6& 6& 6& 6& 6 \\
9 & 9 & 9 & 9& 9& 8& 7& 6& 10& 9 & 6& 6& 6& 6& 6& 7& 7& 7& 7& 7 \\
6 & 6 & 6 & 6& 6& 10& 9& 8& 7& 6 & 9& 9& 9& 9& 9& 8& 8& 8& 8& 8 \\
8 & 8 & 8 & 8& 8& 7& 6& 10& 9& 8 & 7& 7& 7& 7& 7& 9& 9& 9& 9& 9 \\
10 & 10 & 10& 10& 10& 9& 8& 7& 6& 10 & 10& 10& 10& 10& 10& 10& 10& 10& 10& 10 \\
\end{array}\right]\]
\end{example}

\begin{remark}
\textup{All of the invariants described in this paper are defined for virtual
knots and links as well as classical knots and links via the usual method
of ignoring virtual crossings.}
\end{remark}

\begin{definition}
\textup{For each of the invariants described above, the unlink with
$c$ components $U_c$ can have values which look nontrivial. Thus, to aid in
identifying nontrivial values of the invariants, for each invariant
$\phi^{\ast}_{(X,B)}(L)$ we define the corresponding \textit{normalized 
birack invariant} $\bar{\phi}_{(X,B)}^{\ast}(L)$ to be the difference between 
the invariant value on $L$ and the invariant value on the unlink of the
corresponding number of components, i.e.}
\[\bar{\phi}_{(X,B)}^{\ast}(L)=
\phi_{(X,B)}^{\ast}(L)-\phi_{(X,B)}^{\ast}(U_c).\]
\end{definition}

\begin{remark}\textup{The author's 
\texttt{python} code for computing these invariants is available 
at \texttt{www.esotericka.org}. The code represents biracks
as pairs of square matrices and knot or link diagrams as Kauffman-style
signed Gauss codes. The algorithm for finding all birack homomorphisms
uses a working list of partly filled-in image vectors, selecting a blank
entry and trying out all possible values while filling in other entries
using the homomorphism conditions. Any such working image vectors with 
contradictory entries are discarded; any remaining image vectors with 
blanks are appended to the working list, and those with no remaining
blanks are moved to an output list.}
\end{remark}

\section{\large\textbf{Questions}}\label{sec5}

In this section we collect questions for future research.

It is clear that we can combine various enhancement strategies to define
potentially stronger, if more unwieldy, multivariable enhanced invariants.
Computer experiments using \textit{python} code\footnote{available at 
\texttt{www.esotericka.org}} suggest that the writhe enhancement is the most 
effective of the three enhancements considered in this paper for knots and
links with small crossing number using small-cardinality biracks.
Does this remain true when the crossing number or size of the target
birack is increased? 

$(t,s,r)$-biracks generalize Alexander biquandles and $(t,s)$-racks. What
birack definition generalizes bilinear biquandles and symplectic
quandles? What is the generalization of Coxeter racks to the
birack case? What about other new types of biracks?

In both \cite{J} and \cite{M} it is shown that every quandle has a 
presentation as a quotient of a  multiplicative Alexander quandle structure
on the automorphism group of the quandle. Is there an analogous construction
for biracks? The straightforward generalization does not work, but
perhaps there is a non-obvious solution.

Future work will undoubtedly include generalizing the Yang-Baxter cocycle
invariants and quandle cocycle invariants studied in papers like \cite{CJKLS}
and \cite{CN} to the birack case. More generally, though, 
we would like to see additional enhancements of the various counting 
invariants defined. We would also like to see improved algorithms for
computing these invariants quickly. Functorial (as opposed to representational)
invariants of biracks would be of interest as well.

\noindent
\textsc{Department of Mathematical Sciences \\
Claremont McKenna College \\
850 Columbia Ave. \\
Claremont, CA 91711} \\
Email: \texttt{knots@esotericka.org}

\end{document}